\documentclass[12pt,reqno]{amsart}
\usepackage{amssymb}
\usepackage{amsmath, amssymb, amsthm}
\usepackage{mathrsfs,color}

\newcommand{\myred}[1]{\textcolor{red}{#1} }

\newtheorem{theorem}{Theorem}[section]
\newtheorem{assumption}[theorem]{Assumption}
\newtheorem{lemma}[theorem]{Lemma}

\newtheorem{proposition}[theorem]{Proposition}
\newcommand{\dif}{\mathrm{d}}

\newcommand{\di}{\mathrm{div}}
\newcommand{\tr}{\mathrm{tr}} 

\begin{document}
\title[Blowup for UCM model ]{Blowup of solutions for compressible viscoelastic fluid}
\author{Na Wang, Sebastien Boyaval and Yuxi Hu}
 \thanks{\noindent Na Wang,   School of Applied Science,
Beijing Information Science and Technology University, Beijing, 100192, P.R. China, wn\_math@126.com\\
\indent Sebastien Boyaval, LHSV, ENPC, Institut Polytechnique de Paris, EDF R\&D, Chatou, France and Inria, France, sebastien.boyaval@enpc.fr\\
\indent   Yuxi Hu, Department of Mathematics, China University of Mining and Technology, Beijing, 100083  P.R. China, yxhu86@163.com 
}
\begin{abstract}
We prove finite-time blowup of classical solutions for the compressible Upper Convective Maxwell (UCM) viscoelastic fluid system. By establishing a key energy identity and adapting Sideris' method for compressible flows, we derive a Riccati-type inequality for a momentum functional. For initial data with compactly supported perturbations satisfying a sufficiently large condition, all classical solutions lose regularity in finite time. This constitutes the first rigorous blowup result for multidimensional compressible viscoelastic fluids.
\end{abstract}
\maketitle

\section{Introduction}

The dynamics of complex fluids, such as compressible viscoelastic fluids, are governed by the fundamental principles of mass and momentum conservation. For these systems, the conservation of momentum is expressed through the Cauchy stress tensor $\boldsymbol{\sigma}$, which can be decomposed into pressure $p$ and extra stress $\mathbf{T}$ contributions, i.e., $\boldsymbol \sigma = -pI + \mathbf{T}$. The mathematical formulation of the governing equations necessitates the specification of both $p$ and $\mathbf{T}$ to close the system of equations. 

In the context of compressible flows for $(t, \mathbf x)\in \mathbb R^+\times \mathbb R^3$, the mass conservation equation takes the form:
\begin{equation}\label{1.1}
    \frac{D\rho}{Dt} + \rho \nabla \cdot \mathbf{u} = 0,
\end{equation}
where $\rho$ denotes the density field and $\mathbf{u}$ represents the velocity field.  $D/Dt=\partial_t+\mathbf{u}\cdot \nabla$ denotes the material derivative. Meanwhile, the conservation of momentum is given by:
\begin{equation}\label{1.2}
    \rho \frac{D\mathbf{u}}{Dt} = \nabla \cdot \boldsymbol \sigma.
\end{equation}

To address the complexities introduced by fluid compressibility, it is crucial to adopt appropriate constitutive models that accurately describe the behavior of the extra stress $\mathbf{T}$. Unlike the incompressible viscoelastic fluid (see \cite{GuSa90, CM01}) where  $\mathbf T$ satisfy 
\begin{align}\label{in}
   \lambda   \overset{\triangledown}{\mathbf{T}}   + \mathbf{T} =2 \mu_0  \mathbf{D u}  ,
\end{align}
 the rheological equation of state  for $\mathbf T$ should incorporate compressibility effects \cite{BP12}. 
Here, $\mu_0$ is the kinematic viscosity, $\lambda$ is the relaxation time, and $\mathbf{D u}$ is the rate-of-deformation tensor defined as $\mathbf{D u} = \frac{1}{2}(\nabla \mathbf{u} + (\nabla \mathbf{u})^T)$. $\overset{\triangledown}{\mathbf{T}}$ denotes the upper convective derivative (see Oldroyd  \cite{OL50}) , defined by
\begin{equation*}
\overset{\triangledown}{\mathbf{T}} :=\frac{D \mathbf{T}}{Dt}-\nabla \mathbf{u} \cdot \mathbf{T} - \mathbf{T} \cdot (\nabla \mathbf{u})^T.
\end{equation*}
Bollada and Phillips \cite{BP12}, see also  \cite{BE94, EB90}, consider a special Upper Convected Maxwell (UCM) model for a compressible viscoelastic fluid where $\mathbf{T}$ is given by 
\begin{equation}\label{1.3}
    \lambda \left( \overset{\triangledown}{\mathbf{T}}  + (\nabla \cdot \mathbf{u})\mathbf{T} \right) + \mathbf{T} = \mu \left( 2\mathbf{D u} + \left[ \rho \frac{\dif \mu}{\dif \rho} - \mu \right] (\nabla \cdot \mathbf{u})I \right),
\end{equation}
where $\mu(\rho)$ as a function of $\rho$ denotes the dynamic viscosity.
The appearance of additional term $ (\nabla \cdot \mathbf u) \mathbf T$ in \eqref{1.3} is due to the compressible effect, which is originally coming for Truesdell \cite{TR69}. Furthermore, by using the Helmholtz free density given in Beris and Edwards \cite{BE94},  the pressure is supposed to be:
\begin{equation}\label{1.4}
    p = p_0(\rho) + \hat{p}(\rho, \mathbf{T}),
\end{equation}
for an isothermal fluid, where $p_0(\rho)$ represents some equation of state, and $\hat{p}(\rho, \mathbf{T})$ accounts for additional contributions involving the conformation tensor $\mathbf{C}$, related to stress $\mathbf{T}$ via $\mathbf{T} = G(\rho \mathbf{C} - I)$, defined by
\[
\hat p(\rho, T)=\frac{1}{2} \frac{ \partial (\mu/ \rho)}{\partial (1/\rho)} (\mathrm{tr} (\mathbf C)-\ln |\mathbf C|).
\]

In this paper, we assume the dynamic viscosity $\mu$ is proportional to density as $\mu=\mu_0 \rho$, or, equivalently, for constant kinematic viscosity $\mu_0$. Therefore, the system \eqref{1.1}-\eqref{1.4} turn into
\begin{align}\label{1}
\begin{cases}
\rho_t + \nabla \cdot (\rho \mathbf{u}) = 0, \\
\rho \mathbf{u}_t + \rho \mathbf{u} \cdot \nabla \mathbf{u} + \nabla p_0 = \di \mathbf{T}, \\
\lambda \left( \mathbf{T}_t + \mathbf{u} \cdot \nabla \mathbf{T} - \nabla \mathbf{u} \cdot \mathbf{T} - \mathbf{T} \cdot (\nabla \mathbf{u})^T + (\nabla \cdot \mathbf{u}) \mathbf{T} \right) + \mathbf{T} = 2\mu_0 \rho \mathbf{D} \mathbf{u},
\end{cases}
\end{align}
where $p_0(\rho)$ is assumed to satisfy the usual $\gamma$-law, $p_0(\rho)=a \rho^\gamma$ with $a>0$ a constant and $\gamma>1$. 

Let $\mathbf{F}$ be the deformation gradient (see \cite{LLZ05}). Then, by chain rule, $\mathbf F$ satisfies
\begin{align}\label{2}
\mathbf{F}_t + \mathbf{u} \cdot \nabla \mathbf{F} = \nabla \mathbf{u} \cdot \mathbf{F}.
\end{align}
To establish the well-posedness of system \eqref{1},  Boyaval \cite{BY21} introduces a symmetric positive definite matrix $\mathbf A$ via $\mathbf F$, defined by
\begin{equation}\label{1.5}
\mathbf{A} = \mathbf{F}^{-1} \left( \frac{1}{\rho G} \mathbf{T} + \mathbf {I_3} \right) \mathbf{F}^{-T},
\end{equation}
i.e., $\mathbf{T} = \rho G (\mathbf{F} \mathbf{A} \mathbf{F}^T - \mathbf {I_3})$, where $G = \frac{\mu_0}{\lambda}$ is the shear modulus constant. 
It is not difficult to see the matrix $\mathbf{A}$ satisfies the equation
\begin{align}\label{3}
\lambda (\mathbf{A}_t + \mathbf{u} \cdot \nabla \mathbf{A}) + \mathbf{A} = \mathbf{F}^{-1} \mathbf{F}^{-T}.
\end{align}
Therefore, in the sense of classical solutions, system \eqref{1} is equivalent to system 
\begin{align}\label{4}
\begin{cases}
\rho_t + \nabla \cdot (\rho \mathbf{u}) = 0, \\
\rho \mathbf{u}_t + \rho \mathbf{u} \cdot \nabla \mathbf{u} + \nabla p = \nabla \cdot (\rho G(\mathbf{F} \mathbf{A} \mathbf{F}^T - \mathbf {I_3})), \\
\lambda (\mathbf{A}_t + \mathbf{u} \cdot \nabla \mathbf{A}) + \mathbf{A} = \mathbf{F}^{-1} \mathbf{F}^{-T}, \\
\mathbf{F}_t + \mathbf{u} \cdot \nabla \mathbf{F} = \nabla \mathbf{u} \cdot \mathbf{F}
\end{cases}
\end{align}
on noting $\mathbf{T, F, A, \rho}$ are connected by \eqref{1.5} and $p=p_0$ here, since $\partial_\rho\mu_0=0$. 
The local well-posedness of the Cauchy problem for system \eqref{4} with smooth initial data
\begin{equation}\label{1.6}
(\rho, \mathbf{u}, \mathbf{A}, \mathbf{F})(0,x) = (\rho_0, \mathbf{u}_0, \mathbf{A}_0, \mathbf{F}_0)
\end{equation}
is established 
when $\rho_0>0$ and $\rho_0\; {\rm det}(\mathbf{F}_0)$ is a positive constant 
by 
exploiting the system's structure and constructing a convex entropy in \cite{BY21}. 
\myred{The global well-posedness of the Cauchy problem for system \eqref{4} with small initial data was establised in \cite{NBH24}.}
We note that Pan \cite{PA24} established global well-posedness in analytic function spaces for system \eqref{1.1}-\eqref{1.2} coupled with \eqref{in}, under small initial data conditions.

The aim of this paper is to prove that any classical solution of system \eqref{4}--\eqref{1.6} must blow up in finite time, provided the initial data are sufficiently large in an appropriate sense. To our knowledge, besides the one dimensional case \cite{HuRaWa022, HuRa024}, this is the first result concerning blow-up of solutions for the multi-dimensional compressible viscoelastic fluid system. We derive a useful energy identity for system \eqref{4} (or equivalently \eqref{1}; see \eqref{5} below). Utilizing this identity and blow-up methods adapted from Sideris \cite{S85} for compressible Euler equations, we establish a Riccati-type ordinary differential equation and demonstrate the formation of singularities using some technical arguments.

The paper is organized as follows. In Section 2, we  state the assumptions on initial data and present the main results. Section 3 provides a concise and elegant proof of the key theorem.

\section{Statement of the main result}
Before stating our main result,  we first make some assumptions on initial data.
\begin{assumption}\label{assongnull}
1. The initial data $(\rho_0 - \bar \rho, \mathbf{u}_0, \mathbf{A}_0 - \mathbf {I_3}, \mathbf{F}_0 - \mathbf {I_3})$ are compactly supported in 
$B_R := \{\mathbf x\in \mathbb{R}^3 \mid |\mathbf x|\le R\}$ for some $R>0$, where $\bar \rho$ is any positive constant. Without loss of generality, let $\bar \rho = 1$. \\

2. Condition for $\rho_0$:
\begin{align} \label{ass1}
m_0 := \int_{\mathbb{R}^3} (\rho_0 - \bar \rho)\dif \mathbf x \ge 0.
\end{align}

3. Condition for $\mathbf T_0:=\rho_0 G( \mathbf F_0 \mathbf A_0 \mathbf F_0^T-\mathbf {I_3})$:
\begin{align} \label{ass2}
\mathrm{tr} (\mathbf T_0)\ge 0.
\end{align}
\end{assumption}

On the other hand, since \eqref{4} is a hyperbolic system, we have finite propagation speed property, see \cite{Ra015, S85}. 
\begin{proposition}\label{prop1}
Let Assumption \ref{assongnull} hold and suppose that $(\rho, \mathbf{u}, \mathbf{A}, \mathbf{F})$ are $C^1$ solutions to the system \eqref{4} on $[0,T]$. Then, there exists a constant $\sigma$ such that
\[
(\rho(\cdot,t) - \bar \rho, \mathbf{u}(\cdot,t), \mathbf{A}(\cdot,t) - \mathbf {I_3}, \mathbf{F}(\cdot,t) - \mathbf {I_3}) = (0,0,0,0)
\]
on $D(t) = \{\mathbf x\in \mathbb{R}^3 \mid |\mathbf x| > R + \sigma t\}$ for $0 \le t \le T$.
\end{proposition}

Now, we define some useful averaged quantities as follows:
\begin{align}
m(t) &:= \int_{\mathbb{R}^3} (\rho(\mathbf x,t) - \bar \rho)\dif \mathbf x, \\
W(t) &:= \int_{\mathbb{R}^3} \mathbf x \cdot (\rho \mathbf{u})(\mathbf x,t) \dif \mathbf x.
\end{align}

\begin{theorem}\label{th}
Suppose that $(\rho, \mathbf{u}, \mathbf{A}, \mathbf{F})$ are any $C^1$ solution of \eqref{4}--\eqref{1.6}. Then there exists initial data satisfying Assumption \ref{assongnull} and 
\begin{align}\label{11}
W(0) \ge \lambda \left( H_0 + \max \rho_0 \|\mathbf{u}_0\|_{L^2}^2 \right) + \frac{16 R^4 \sigma \pi \max \rho_0}{3},
\end{align}
where $H_0$ is defined in \eqref{h0}, such that the life span $T$ of the $C^1$ solution is finite.
\end{theorem}

\section{Blow up of solutions}

We first present the following two lemmas.

\begin{lemma} \label{le1}
Let $(\rho, \mathbf{u}, \mathbf{A}, \mathbf{F})$ be local solutions to system \eqref{4} and $\mathbf{T} = \rho G(\mathbf{F} \mathbf{A} \mathbf{F}^T - \mathbf {I_3})$ satisfies \eqref{1}$_3$. Then, the following identity holds
\begin{multline}
\frac{\dif}{\dif t} \int_{\mathbb{R}^3} \left( \frac{\rho }{2} |\mathbf{u}|^2
+ \frac{a}{\gamma-1} (\rho^\gamma - 1 - \gamma(\rho - 1)) + \frac{\mu_0}{\lambda} (\rho \ln \rho - \rho + 1)
+ \frac{\tr (\mathbf{T})}{2} \right) \dif \mathbf x
\\
+ \frac{1}{\lambda} \int_{\mathbb{R}^3} \frac{\tr (\mathbf{T})}{2} \dif \mathbf x
= 0.
\label{5}
\end{multline}
\end{lemma}
\begin{proof}
Multiplying \eqref{4}$_2$ by $\mathbf{u}$, using the mass equation \eqref{4}$_1$, and integrating the result, we get
after integration by part (recall the solution $\mathbf{u}$ has finite support):
\begin{equation*}
\frac{\dif}{\dif t} \int_{\mathbb{R}^3} \left( \frac{\rho}{2} |\mathbf{u}|^2 
+ \tr (\mathbf{T} (\nabla \mathbf{u})^T)\right) \dif \mathbf x
= \int_{\mathbb{R}^3}  p \nabla\cdot\mathbf{u}  \dif \mathbf x.
\end{equation*}
Taking the trace of equation \eqref{1}$_3$, using the mass equation \eqref{4}$_1$ and integrating the result, one has:
\begin{align*}
\frac{\dif}{\dif t} \int_{\mathbb{R}^3} \frac{\tr (\mathbf{T})}{2} \dif \mathbf x - \int_{\mathbb{R}^3} \tr (\mathbf{T} (\nabla \mathbf{u})^T) \dif \mathbf x
+ \frac1\lambda \int_{\mathbb{R}^3} \frac{\tr (\mathbf{T})}{2}   \dif \mathbf x
= \int_{\mathbb{R}^3} \frac{\mu_0 \rho}{\lambda} \nabla\cdot \mathbf{u} \dif \mathbf x.
\end{align*}
Combining the above estimates, and using the mass equation again
(with the fact that the solution $\mathbf{u}$ has finite support)
we obtain the desired result.
\end{proof}

\begin{lemma}\label{le2}
For $0 \le t \le T$, we have
\begin{align}
m(t) = \int_{\mathbb{R}^3} (\rho(\mathbf x,t) - \bar \rho) \dif \mathbf x &\ge 0, \label{7} \\
\int_{\mathbb{R}^3} (p(\rho) - p(\bar \rho)) \dif \mathbf x &\ge 0, \label{8} \\
\int_0^t \int_{\mathbb{R}^3} \tr (\mathbf{T}) \dif \mathbf x \dif t &\le
\lambda \left( H_0 + {\max \rho_0} \|\mathbf{u}_0\|_{L^2}^2\right),  \label{9}
\end{align}
where
\begin{align}\label{h0}
H_0 = \int_{\mathbb{R}^3} \left(
\frac{2a}{\gamma-1} (\rho_0^\gamma - 1 - \gamma(\rho_0 - 1)) + \frac{2\mu_0}{\lambda} (\rho_0 \ln \rho_0 - \rho_0 + 1)
+ \tr (\mathbf{T}_0) \right) \dif \mathbf x.
\end{align}
\end{lemma}
\begin{proof}
Integrating the mass equation \eqref{1}$_1$ over $\mathbb{R}^3$, using Proposition \ref{prop1} and Assumption \ref{assongnull}, we immediately get
$$
m(t) = \int_{\mathbb{R}^3} (\rho(x,t) - \bar \rho) \dif \mathbf x = \int_{\mathbb{R}^3} (\rho_0 - \bar \rho) \dif \mathbf x \ge 0.
$$

Moreover, by Jensen's inequality and \eqref{7}, we have
\begin{align*}
\int_{B(t)} p(\rho) \dif \mathbf x &= a\int_{B(t)} \rho^\gamma (x,t)\dif \mathbf x\\
&\ge a |B(t)|^{1-\gamma} \left( \int_{B(t)} \rho(x,t) \dif \mathbf x \right)^\gamma\\
&\ge a |B(t)|^{1-\gamma} \left(\int_{B(t)} \bar \rho \dif \mathbf x \right)^\gamma\\
&= a |B(t)| \bar \rho^\gamma =a \int_{B(t)} \bar \rho^\gamma \dif \mathbf x=\int_{B(t)} p(\bar \rho) \dif \mathbf x.
\end{align*}
where $B(t)=\{x:|x|\le R+\sigma t\}=D^c(t)$ and $|B(t)|$ denotes the volume of $B(t)$. This implies \eqref{8}.   
On the other hand, by Taylor expansion, $\rho^\gamma-1-\gamma (\rho-1)\ge 0$ 
and $\rho \ln \rho-\rho+1\ge 0$ 
whatever $\rho\ge0$, and using Lemma \ref{le1} we have
\begin{align*}
\int_{\mathbb{R}^3} \tr (\mathbf{T}) \dif \mathbf x + \frac1\lambda\int_0^t \int_{\mathbb{R}^3} \tr (\mathbf{T}) \dif \mathbf x \dif t \le
H_0 + {\max \rho_0}\|\mathbf{u}_0\|_{L^2}^2
\end{align*}
after integrating \eqref{5} over $(0,t)$, i.e. for all times $t\ge 0$
$$
\frac{\dif}{\dif t} \left( e^{\frac{t}\lambda} \int_0^t \int_{\mathbb{R}^3} \tr (\mathbf{T}) \dif \mathbf x \dif t -
e^{\frac{t}\lambda}(H_0 + {\max \rho_0}\|\mathbf{u}_0\|_{L^2}^2)\lambda  \right) \le 0
$$
which immediately yields the inequality \eqref{9}. 
\end{proof}
 

{\bf Proof of Theorem \ref{th}}:
Using equation $\eqref{1}_2$ and integrating by part, we get
\begin{align*}
W^\prime(t)&=\int_{\mathbb R^3} \mathbf x \cdot (\rho \mathbf{u})_t\dif \mathbf x\\
&=\int_{\mathbb R^3} \mathbf x \cdot \left(-\nabla\cdot (\rho \mathbf{u}\otimes \mathbf{u})-\nabla p+\nabla \cdot \mathbf{T}\right)\dif \mathbf x\\
&=\int_{\mathbb R^3} \left[\rho |\mathbf{u}|^2+3(p(\rho)-p(\bar \rho))-\mathrm{tr} (\mathbf{T})\right]\dif \mathbf x.
\end{align*}
Therefore, Lemma \ref{le2} implies
\begin{align}
W^\prime(t)\ge \int_{\mathbb R^3} \rho |\mathbf{u}|^2 \dif \mathbf x- \int_{\mathbb R^3} \mathrm{tr} (\mathbf{T})\dif \mathbf x.
\end{align}
On the other hand, using Cauchy-Schwarz inequality, we get
\begin{align*}
W^2(t)=\left( \int_{B(t)} \mathbf x \cdot (\rho \mathbf{u}) \dif \mathbf x \right)^2
\le\left( \int_{B(t)} |\mathbf x|^2 \rho \dif \mathbf x \right) \cdot \left(\int_{B(t)} \rho |\mathbf{u}|^2 \dif \mathbf x \right),
\end{align*}
while, in virtue of $\eqref{7}$, we have
\begin{align*}
\int_{B(t)} |\mathbf x|^2 \rho \dif \mathbf x &\le (R+\sigma t)^2 \int_{B(t)} \rho \dif \mathbf x \\
&=(R+\sigma t)^2 \left( m(0)+\int_{B(t)} \bar \rho \dif \mathbf x \right)\\
&=(R+\sigma t)^2 \left( \int_{B(t)} (\rho_0-\bar \rho) \dif \mathbf x +\int_{B(t)} \bar \rho \dif \mathbf x \right)\\
&\le \frac{4}{3}\pi  (R+\sigma t)^5 \max \rho_0.
\end{align*}
Combining the above inequalities, we get
\begin{align}\label{12}
W^\prime(t)\ge \left( \frac{4}{3}\pi (R+\sigma t)^5 \max \rho_0 \right)^{-1} W^2(t) - \int_{\mathbb R^3} \mathrm{tr} (\mathbf{T})\dif \mathbf x.
\end{align}

Let
$$
c_2:=\frac{\sigma}{R},\qquad c_3:= \frac{3}{4 \pi \max \rho_0 R^5}.
$$
Then, integrating \eqref{12} over $(0,t)$ and using \eqref{9}, we derive that
\begin{align}
W(t)\ge \int_0^t \frac{c_3}{(1+c_2 s)^5} W^2(s) \dif s +U_0,
\end{align}
where  $U_0=W(0)-\lambda\left( H_0+
\max \rho_0 \|\mathbf{u}_0\|_{L^2}^2\right)$.
Note that,  assumption \eqref{11} is equivalent to
\begin{align}\label{14}
U_0>\frac{4c_2}{c_3}=\frac{16 R^4 \sigma \pi \max \rho_0}{3}>0.
\end{align}

Let
$$
V(t)=\int_0^t \frac{c_3}{(1+c_2 s)^5} W^2(s)\dif s+U_0,
$$
then,
\begin{align}
\frac{\dif}{\dif t} V(t)=\frac{c_3}{(1+c_2 t)^5} W^2(t)\ge \frac{c_3}{(1+c_2 t)^5} V^2(t),
\end{align}
for which we have
\begin{align}
\frac{1}{U_0}=\frac{1}{V(0)}\ge \frac{1}{V(0)}-\frac{1}{V(T)} \ge \frac{c_3}{4c_2}-\frac{c_3}{4 c_2(1+c_2 T)^4},
\end{align}
which means that $T$ cannot be arbitrarily large without contradicting \eqref{14}.

Thus, the proof will be finished if we can show there exists $\mathbf{u}_0$ such that \eqref{14} (or equivalently, \eqref{11}) holds and Assumption \ref{assongnull} is satisfied. Let (cp. \cite{HuRa023,HuRaWa022})
\begin{align} \label{5.12}
\tilde v(r)=
\begin{cases}
L \cos\left(\frac{\pi}{2}(r-1)\right), & r\in [0,1],\\
L, & r\in(1, R-1],\\
\frac{L}{2} \cos ( \pi(r-R+1))+\frac{L}{2}, & r\in (R-1, R],\\
0, & r\in(R, +\infty),
\end{cases}
\end{align}
where $L$ is a positive constant to be determined later. $\tilde v$ is not in $H^3(\mathbb{R}_+)$, but we can think of $\tilde v$ being smoothed around the singular points $r=1, R-1, R$ and put to zero around $r=0$, yielding a function $v$, with $v(r)\ge L$ when $r\in (2,R-2)$
and $\|v\|_{L^2}\leq 2 \|\tilde v\|_{L^2}$. We choose
$$
\mathbf{u}_0(\mathbf x):=v(|\mathbf x|)\frac{\mathbf x}{|\mathbf x|}.
$$

Inequalities \eqref{ass1} and \eqref{ass2} can easily be satisfied by choosing $\rho_0 > \bar \rho =1$ and $\mathbf A_0$ such that $\mathrm{tr}( \mathbf F_0 \mathbf A_0 \mathbf F_0^T-\mathbf {I_3})\ge 0$ with $\mathbf F_0=\mathbf{I_3}$. Then, letting $R\ge 5$ such that $(R-2)^4-2^4>R^4/32$ e.g., it holds 
\begin{align*}
W(0)\equiv \int_{\mathbb R^3} \mathbf x \cdot \mathbf{u}_0(\mathbf x) \rho_0(\mathbf x) \dif \mathbf x &=\int_{\mathbb R^3} \rho_0(\mathbf x) v(|\mathbf x|) |\mathbf x|\dif \mathbf x \\
&\ge \min \rho_0 \int_{B_0(R)} v(|\mathbf x|) |\mathbf x|\dif \mathbf x \\
&\ge \min \rho_0 \int_0^R v(r) r\cdot 4\pi r^2 \dif r \\
&\ge \min \rho_0 \int_2^{R-2} L\cdot 4\pi r^3 \dif r \ge \frac{\pi \min \rho_0}{32} L R^4.
\end{align*}
We choose $L$ (independent of $R$) sufficiently large such that
\begin{align}\label{19}
\frac{\pi \min \rho_0}{64} L \ge \frac{16 \sigma \pi \max \rho_0}{3}.
\end{align}

Now, after having chosen $L$ large enough, fix $L$. On the other hand, since $\|\mathbf{u}_0\|_{L^2}^2 \le 4 L^2 \cdot \frac{4\pi}{3}R^3$, we choose $R$ sufficiently large such that
\begin{align}\label{20}
\frac{\pi \min \rho_0}{64} L R^4 \ge \lambda\left( H_0+ \max \rho_0 4 L^2 \cdot \frac{4\pi}{3}R^3\right).
\end{align}
Combining \eqref{19} and \eqref{20}, the assumption \eqref{14} is satisfied.
This finishes the proof of Theorem \ref{th}.\hfill $\Box$.

\end{document}